\documentclass[oneside,english]{amsart}
\usepackage[T1]{fontenc}
\usepackage[latin9]{inputenc}
\usepackage[a4paper]{geometry}
\usepackage{amstext}
\usepackage{amsthm}
\usepackage{amssymb}
\usepackage{color}

\makeatletter
\numberwithin{equation}{section}
\numberwithin{figure}{section}
\theoremstyle{plain}
\newtheorem{thm}{\protect\theoremname}
  \theoremstyle{plain}
  \newtheorem{lem}{\protect\lemmaname}

\DeclareMathOperator{\supp}{supp}

\makeatother

\usepackage{babel}
  \providecommand{\lemmaname}{Lemma}
\providecommand{\theoremname}{Theorem}

\begin{document}

\title{ A lower bound for $A_{p}$ exponents for some weighted weak-type inequalities}

\author[Carlos P\'erez, Israel P. Rivera-R\'ios]{Carlos P\'erez, Israel P. Rivera-R\'ios}
\address[C. P\'erez]{Departamento de Matem\'aticas, Universidad del Pa\'is Vasco UPV/EHU,
IKERBASQUE, Basque Foundation for Science, and BCAM,
Basque Center for
Applied Mathematics, Bilbao, Spain.}
\email{carlos.perezmo@ehu.es}
\address[I. P. Rivera-R\'ios]{Departamento de Matem\'aticas, Universidad del Pa\'{\i}s Vasco UPV/EHU and BCAM,}
\email{petnapet@gmail.com}

\thanks{Both authors are supported by the Basque Government through the BERC 2014-2017 program and by Spanish Ministry of Economy and Competitiveness MINECO through BCAM Severo Ochoa excellence accreditation SEV-2013-0323. and through the project MTM2014-53850-P. I.P.R-R. is also supported by Spanish Ministry of Economy and Competitiveness MINECO through the project MTM2012-30748.}

\begin{abstract}
 We give a weak-type counterpart of the main result in \cite{LPR} which allows to provide a lower bound for the exponent of the $A_{p}$ constant in terms of the behaviour of the unweighted inequalities
when $p\rightarrow\infty$ and when $p\rightarrow1^{+}$.  We also provide some applications to classical operators.

\end{abstract}

\maketitle

\section{Introduction and main results}

The purpose of this paper is to give a weak-type counterpart of the main result in \cite{LPR}.   If $T$ is an operator which satisfies a weak type bound like 
\begin{equation}\label{model-quantitative}
\|T\|_{ L^{p}(w)  \to L^{p,\infty}(w)   }\le c\, [w]^{\beta}_{A_p} \qquad w \in A_{ p},
\end{equation}
with $\beta>0,$ then we will show in Theorem \ref{Thm:LowerBoundWE}  that the  optimal lower bound  for $\beta$ is related to the  asymptotic behaviour of the unweighted $L^p$ norm $\|T\|_{L^{p}(\mathbb{R}^n)  \rightarrow L^{p,\infty}(\mathbb{R}^n) }$ as $p$ goes to $1$ and $+\infty$. We recall that a weight $w$, namely a non-negative locally integrable
function, belongs to the $A_{p}$ class of Muckenhoupt if 
\[
\begin{split}[w]_{A_{p}} & =\sup_{Q}\left(\frac{1}{|Q|}\int_{Q}w\right)\left(\frac{1}{|Q|}\int_{Q}w^{\frac{1}{p-1}}\right)^{p-1}<\infty\qquad(1<p<\infty)\\
[w]_{A_{1}} & =\sup_{x\in\mathbb{R}^{n}}\frac{Mw(x)}{w(x)}<\infty
\end{split}
\]
where $M$ stands for the Hardy-Littlewood maximal function, namely
\[Mf(x)=\sup_{Q\ni x}\frac{1}{|Q|}\int_Q|f|\]
and each $Q$ is a cube with its sides parallel to the axis.

The $A_p$ conditions characterize the weighted $L^p$ boundedness of the maximal function, 
namely $w\in A_p$ if and only if the corresponding estimate
\begin{equation}
\begin{split}\|Mf\|_{L^{p}(w)} & \leq c_{w,n,p}\|f\|_{L^{p}(w)}\qquad(1<p<\infty)\\
\|Mf\|_{L^{1,\infty}(w)} & \leq c_{w,n}\|f\|_{L^{1}(w)}   \qquad  w\in A_1
\end{split}
\label{eq:MaxQual}
\end{equation}
holds, where $c_{w,n,p}$ is a constant that depends on the weight,
on the dimension $n$ and on $p$. Since Muckenhoupt's seminal work, many authors
such as Wheeden, Hunt, Coifman or Fefferman, got involved in the study
of weighted estimates, providing interesting results such for singular
integrals as well. 

In the last decade, one of the main problems in Harmonic Analysis has been the study of sharp norm inequalities for some of the classical operators on weighted Lebesgue spaces $L^p(w), \, 1<p<\infty$. Some examples of those kind of results include include the Hardy--Littlewood maximal operator, the Hilbert transform and more generally Calder\'on-Zygmund operators (C--Z operators). 
Given any of these operators $T$, the first part of this problem is to look for quantitative bounds of the norm $\|T\|_{L^p(w)}$ in terms of  the $A_p$ constant of the weight, namely an estimate like \eqref{model-quantitative}. The following step is to establish the sharp dependence, typically with respect to the power of $[w]_{A_p}$,   i.e. the optimality of $\beta$ in \eqref{model-quantitative}. In recent years, the answer to this last question has let a fruitful activity and development of new tools in Harmonic Analysis. Firstly, in the early 90s, Buckley \cite{B} identified the sharp exponent in the case of the Hardy--Littlewood maximal function, i.e.,
\begin{equation}\label{estMax}
\begin{split}\|Mf\|_{L^{p}(w)} & \leq c_{n,p}[w]_{A_{p}}^{\frac{1}{p-1}}\|f\|_{L^{p}(w)}\qquad(1<p<\infty),\\
\|Mf\|_{L^{p,\infty}(w)} & \leq c_{1}[w]_{A_{p}}^{\frac{1}{p}}\|f\|_{L(w)}\qquad(1\leq p<\infty).
\end{split}
\end{equation}

However, Buckley's work was not very influential  initially.  Quantitative estimates did not become
an important topic until the work of Astala, Iwaniec and Saksman \cite{AIS}
in which they proved that the solvavility of the Beltrami equation relied
upon the linear dependence on the $A_{2}$ constant of the Beurling
transform, namely on the following estimate
\[
\|Bf\|_{L^{2}(w)}\leq c_{n}[w]_{A_{2}}\|f\|_{L^{2}(w)}.
\]
That estimate was shortly after proved to be true by Petermichl and
Volberg \cite{PV}, and can be considered the beginning of the ``quantitative
estimates era''. Several authors have made more than interesting
contributions to this topic. Especially, the proof
of the $A_{2}$ conjecture \cite{H} (improved in \cite{HPAinfty}) and the quest for simpler proofs
has led to developments such as the sparse domination theory that
probably were unconceiveable years ago.

In this work we provide a criterium to decide the sharp dependence
of the $A_{p}$ constant for the weak-type $(p,p)$ estimate based on the behaviour of  $\|T\|_{L^{p}(\mathbb{R}^n)  \rightarrow L^{p,\infty}(\mathbb{R}^n) }  $ when $p\to 1$ and $p\to \infty$. The main result is the following. 

\begin{thm}
\label{Thm:LowerBoundWE}Given an operator $T$ such that for some
$1<p_{0}<\infty$ and for any $w\in A_{p_{0}}$ 
\begin{equation}\label{hyp}
\|T\|_{L^{p_{0},\infty}(w)}\leq c\left[w\right]_{A_{p_{0}}}^{\beta}
\end{equation}
then 
\[
\beta\geq\max\left\{ \gamma_{T};\ \frac{\alpha_{T}}{p_{0}-1}\right\} 
\]
where
\[
\alpha_{T}=\sup\left\{ \alpha\geq0\ :\ \forall\varepsilon>0\ \limsup_{p\rightarrow1^{+}}(p-1)^{\alpha-\varepsilon}\|T\|_{L^{p}\rightarrow L^{p,\infty}}=\infty\right\} 
\]
and 
\[
\gamma_{T}=\sup\left\{ \gamma\geq0\ :\ \forall\varepsilon>0\ \limsup_{p\rightarrow\infty}\frac{\|T\|_{L^{p}\rightarrow L^{p,\infty}}}{p^{\gamma-\varepsilon}}=\infty\right\} .
\]
\end{thm}
To apply the preceding result we need to provide sharp unweighted
estimates in terms of $p$ and $p'$. We gather such estimates for
some cases of interest in Lemma \ref{LemUnweighted}. Now we present those operators. 

We say that $T$ is a Calderón-Zygmund operator if $T$ is bounded
on $L^{2}$ and admits the following representation for $f\in\mathcal{C}_{c}^{\infty}$
\[
Tf(x)=\int_{\mathbb{R}^{n}}K(x,y)f(y)dy\qquad x\not\in\supp f
\]
where $K:\mathbb{R}^{n}\times\mathbb{R}^{n}\setminus\left\{ (x,x)\,:\,x\in\mathbb{R}^{n}\right\} \rightarrow\mathbb{R}$
is a kernel satisfying the following properties
\begin{enumerate}
\item $|K(x,y)|\leq\frac{C_{K}}{|x-y|^{n}}$
\item $\left|K(x,y)-K(x',y)\right|\leq C\left(\frac{|x-x'|}{|x-y|}\right)^{\delta}\frac{1}{|x-y|^{n}}$
where $|x-x'|\leq\frac{1}{2}|x-y|$ for some $\delta>0$.
\end{enumerate}
Relying upon the preceding definitions, we have that given $b\in BMO$
and $T$ a Calderón-Zygmund operator we define the commutator  $[b,T]$
by
\[
[b,T]f(x)=b(x)Tf(x)-T(bf)(x).
\]
We recall that $b\in BMO$ if 
\[
\|b\|_{BMO}=\sup_{Q}\frac{1}{|Q|}\int_{Q}|b-b_{Q}|dx<\infty.
\]

At this point we are in the position to state the lemma that we announced before.
\begin{lem} \label{LemUnweighted}Let $1<p<\infty$, $T$ a Calderón-Zygmund operator
and $b\in BMO$. Then there exist constants $c_{i}>0$ such that the
following estimates hold
\[
\begin{split} &\|T\|_{L^{p}\longrightarrow L^{p,\infty}}\leq c_{1}p\\
 &\|[b,T]\|_{L^{p}\longrightarrow L^{p,\infty}}\leq c_{2}p'p^{2}\\
 & c_{3}(p')^{k-1}\leq\|M^{k}\|_{L^{p}\rightarrow L^{p,\infty}}\leq c_{4}(p')^{k-1}
\end{split}
\]
where $M^k$ stands for $M\circ\stackrel{(k)}{\dots}\circ M$. 

On the other hand, if $H$ is the Hilbert transform and $b(x)=\log|x|$,  we also have that there exist constants $c_i>0$ such that
\[
\begin{split} & c_{5}p\leq\|H\|_{L^{p}\longrightarrow L^{p,\infty}}\\
 & c_{6}\max\left\{ p',p^{2}\right\} \leq\|[b,H]\|_{L^{p}\longrightarrow L^{p,\infty}}
\end{split}
\]

\end{lem}

Combining some known estimates in the literature and the preceding results
we obtain the following result.
\begin{thm}
\label{Thm:ApBounds}Let $1<p<\infty$, $k$ a positive integer, $T$
a Calderón-Zygmund operator and $b\in BMO$. Then 

\begin{enumerate}
\item \textup{$\|T\|_{L^{p,\infty}(w)}\leq c\left[w\right]_{A_{p}}$ and
the exponent of the $A_{p}$ constant is sharp.}
\item $\|[b,T]\|_{L^{p,\infty}(w)}\leq c\left[w\right]_{A_{p}}^{\rho_{p}}$
where have that $\max\left\{ 2,\frac{1}{p-1}\right\} \leq\rho_{p}\leq\max\left\{ 2,p'\right\} $.
\item $\|M^{k}\|_{L^{p}\rightarrow L^{p,\infty}}\leq c[w]_{A_{p_{0}}}^{\eta_{p_{0}}}$
with $\frac{k-1}{p_{0}-1}\leq\eta_{p_{0}}\leq\frac{1}{p}+\frac{k-1}{p_{0}-1}$
\end{enumerate}
\end{thm}
At this point some remarks are in order. We observe that our method
is completely satisfactory in the case of Calderón-Zygmund operators,
contrarily to what happens in the case of the commutator $[b,T]$
and for the maximal function $M$ and its iterations $M^{k}$. In the case of the maximal function, that fact is not a suprise. 
The information that the method provides comes from the relationship of the boundedness
constant of the operator with the exponents  of $p$ and $p'$ and in this case both exponents are zero, thus, the method cannot provide any kind of information. In the case of the commutator it is not clear whether the upper bound can be improved or the lower one should be larger.

The rest of the paper is organized as follows. Section \ref{sec:ProoTwmLBWE}
is devoted to the proof of Theorem \ref{Thm:LowerBoundWE}. Lemma \ref{LemUnweighted}
is established in Section \ref{sec:LemUnweighted}. We end up this
paper with the proof of Theorem \ref{Thm:ApBounds} which is presented
in Section \ref{sec:ThmApBounds}.
\section{Proof of Theorem \ref{Thm:LowerBoundWE}\label{sec:ProoTwmLBWE}}
As we mentioned before, we will adapt here the main arguments from \cite{LPR}  which in turn is based on ideas from \cite{CMP} and \cite{Duo}.   

Firstly we prove that $\frac{\alpha_{T}}{p_{0}-1}\leq\beta$. If $\alpha_{T}=0$
there's nothing to prove, so let us assume that $\alpha_{T}>0$. We
define then the following Rubio de Francia algorithm (see for instance \cite{GCRdF})
\[
R_{p}f(x)=\sum_{k=0}^{\infty}\frac{M^{k}f(x)}{2^{k}\|M\|_{L^{p}}^{k}}.
\]
$R$ satisfies the following properties
\begin{enumerate}
\item $h\leq R_{p}h$
\item $\|R_{p}h\|_{L^{p}}\leq2\|h\|_{L^{p}}$
\item $R_{p}h\in A_{1}$. Furthermore $[R_{p}h]_{A_{1}}\leq 2 \, \|M\|_{L^{p}}.$
\end{enumerate}
We have that 
\[
\begin{split}\|Tf\|_{L^{p,\infty}} & =\sup_{\lambda>0}\lambda\left|\left\{ x\in\mathbb{R}^{n}\,:\,|Tf(x)|>\lambda\right\} \right|^{\frac{1}{p}}\\
 & =\sup_{\lambda>0}\lambda\left(\int_{\left\{ x\in\mathbb{R}^{n}\,:\,|Tf(x)|>\lambda\right\} }\right)^{\frac{1}{p}}\\
 & =\sup_{\lambda>0}\lambda\left(\int_{\left\{ x\in\mathbb{R}^{n}\,:\,|Tf(x)|>\lambda\right\} }\left(R_{p}f\right)^{-(p_{0}-p)\frac{p}{p_{0}}}\left(R_{p}f\right)^{(p_{0}-p)\frac{p}{p_{0}}}dx\right)^{\frac{1}{p}}\\
 & \leq\sup_{\lambda>0}\lambda\left(\int_{\left\{ x\in\mathbb{R}^{n}\,:\,|Tf(x)|>\lambda\right\} }\left(R_{p}f\right)^{-(p_{0}-p)}dx\right)^{\frac{1}{p_{0}}}\left(\int_{\left\{ x\in\mathbb{R}^{n}\,:\,|Tf(x)|>\lambda\right\} }\left(R_{p}f\right)^{p}dx\right)^{\frac{p_{0-p}}{pp_{0}}}\\
 & \leq\|Rf\|_{L^{p}(\mathbb{R}^{n})}^{\frac{p_{0}-p}{p_{0}}}\sup_{\lambda>0}\lambda\left(\int_{\left\{ x\in\mathbb{R}^{n}\,:\,|Tf(x)|>\lambda\right\} }\left(R_{p}f\right)^{-(p_{0}-p)}dx\right)^{\frac{1}{p_{0}}}\\
 & =\|f\|_{L^{p}(\mathbb{R}^{n})}^{\frac{p_{0}-p}{p_{0}}}\|Tf\|_{L^{p_{0},\infty}\left(\left(R_{p}f\right)^{-(p_{0}-p)}\right)}
\end{split}
\]
We observe that this estimate holds for both $p<p_0$ or $p>p_0$. Now, if we fix $1<p<p_0$ it was established in \cite{LPR} that
$\left[\left(R_{p}f\right)^{-(p_{0}-p)}\right]_{A_{p_{0}}}\leq c_n \|M\|_{L^p}^{p_0-p}$. Taking that into account and applying the hypothesis \eqref{hyp}  we have that
\[
\begin{split}\|f\|_{L^{p}(\mathbb{R}^{n})}^{\frac{p_{0}-p}{p_{0}}}\|Tf\|_{L^{p_{0},\infty}\left(\left(Rf\right)^{-(p_{0}-p)}\right)} & \leq c\left[\left(R_{p}f\right)^{-(p_{0}-p)}\right]_{A_{p_{0}}}^{\beta}\|f\|_{L^{p}(\mathbb{R}^{n})}^{\frac{p_{0}-p}{p_{0}}}\|f\|_{L^{p_{0}}(\left(R_{p}f\right)^{-(p_{0}-p)})}\\
 & \leq c\left[\left(R_{p}f\right)^{-(p_{0}-p)}\right]_{A_{p_{0}}}^{\beta}\|f\|_{L^{p}(\mathbb{R}^{n})}.
\end{split}
\]
Then, 
\[
\|Tf\|_{L^{p,\infty}}\leq c\|M\|_{L^{p}}^{\beta\left(p_{0}-p\right)}\|f\|_{L^{p}(\mathbb{R}^{n})}\qquad1<p<p_{0}
\]
Now we recall that 
\[
\|M\|_{L^{p}(\mathbb{R}^{n})}\leq c\frac{1}{p-1}
\]
Then for $p$ close to $1$ we get
\[
\|T\|_{ L^{p} \rightarrow L^{p,\infty}   }  \leq c\left(p-1\right)^{-\beta(p_{0}-p)}\leq c\left(p-1\right)^{-\beta(p_{0}-1)}
\]
Since $\alpha_{T}>0$ if $\alpha_{T}-\varepsilon>0$, multiplying
by $(p-1)^{\alpha_{T}-\varepsilon}$, and taking $\limsup$, by the
definition of $\alpha_{T}$
\[
\begin{split}\infty & =\begin{split}\limsup_{p\rightarrow1\text{\textsuperscript{+}}} & \|T\|_{L^{p,\infty}\rightarrow L^{p}}(p-1)^{\alpha_{T}-\varepsilon}\end{split}
\\
 & \leq\limsup_{p\rightarrow1\text{\textsuperscript{+}}}c\left(p-1\right)^{-\beta(p_{0}-1)+\alpha_{T}-\varepsilon}
\end{split}
\]
we have that 
\[
-\beta(p_{0}-1)+\alpha_{T}-\varepsilon<0\iff\frac{\alpha_{T}-\varepsilon}{p_{0}-1}<\beta
\]
And taking inf in $\varepsilon$, 
\[
\frac{\alpha_{T}}{p_{0}-1}\leq\beta
\]
Let us prove now that $\gamma_{T}\leq\beta$. We follow the same extrapolation
ideas, but now we use the dual space $L^{p'}(\mathbb{R}^{n})$. Fix
$p>p_{0}$ and $f\in L^{p}(\mathbb{R}^{n})$. Firstly we observe that
our hypothesis is equivalent to
\[
tw\left(\left\{ x\in\mathbb{R}^{n}\,:\,|Tf(x)|>t\right\} \right)^{\frac{1}{p_{0}}}\leq c\left[w\right]_{A_{p_{0}}}^{\beta}\|f\|_{L^{p_{0}}}\qquad t>0.
\]
Now
\[
\begin{split}\|Tf\|_{L^{p,\infty}(\mathbb{R}^{n})} & =\sup_{\lambda>0}\lambda\left(\int_{\left\{ x\in\mathbb{R}^{n}\,:\,|Tf(x)|>\lambda\right\} }\right)^{\frac{1}{p}}\\
 & =\sup_{\lambda>0}\lambda\left\Vert \chi_{\left\{ x\in\mathbb{R}^{n}\,:\,|Tf(x)|>\lambda\right\} }\right\Vert _{L^{p}(\mathbb{R}^{n})}
\end{split}
\]
Now by duality for each $\lambda>0$ we can find $h_{\lambda}\in L^{p'}(\mathbb{R}^{n})$
, $h_{\lambda}\geq0$, $\left\Vert h_{\lambda}\right\Vert _{L^{p'}(\mathbb{R}^{n})}=1$
such that 
\[
\begin{split} & \lambda\left\Vert \chi_{\left\{ x\in\mathbb{R}^{n}\,:\,|Tf(x)|>\lambda\right\} }\right\Vert _{L^{p}(\mathbb{R}^{n})}\\
 & =\lambda\int_{\left\{ x\in\mathbb{R}^{n}\,:\,|Tf(x)|>\lambda\right\} }h_{\lambda}\\
 & \leq\lambda\int_{\left\{ x\in\mathbb{R}^{n}\,:\,|Tf(x)|>\lambda\right\} }\left(R_{p'}h_{\lambda}\right)^{\frac{p-p_{0}}{p_{0}(p-1)}}h_{\lambda}^{\frac{p(p_{0}-1)}{p_{0}(p-1)}}dx\\
 & \leq\lambda\left(\int_{\left\{ x\in\mathbb{R}^{n}\,:\,|Tf(x)|>\lambda\right\} }\left(R_{p'}h_{\lambda}\right)^{\frac{p-p_{0}}{p-1}}\right)^{\frac{1}{p_{0}}}\left(\int_{\mathbb{R}^{n}}h_{\lambda}^{p'}dx\right)^{\frac{1}{p_{0}'}}\\
 & =\lambda\left(\int_{\left\{ x\in\mathbb{R}^{n}\,:\,|Tf(x)|>\lambda\right\} }\left(R_{p'}h_{\lambda}\right)^{\frac{p-p_{0}}{p-1}}\right)^{\frac{1}{p_{0}}}
\end{split}
\]
Now we observe that $w=\left(R'h_{\lambda}\right)^{\frac{p-p_{0}}{p-1}}$
is an $A_{p_{0}}$ weight. Using hypothesis this yields, 
\[
tw\left(\left\{ x\in\mathbb{R}^{n}\,:\,|Tf(x)|>t\right\} \right)^{\frac{1}{p_{0}}}\leq c\left[w\right]_{A_{p_{0}}}^{\beta}\|f\|_{L^{p_{0}}}\qquad t>0
\]
In particular that inequality holds for $t=\lambda$. Then we have
that 
\[
\begin{split} & \lambda\left(\int_{\left\{ x\in\mathbb{R}^{n}\,:\,|Tf(x)|>\lambda\right\} }\left(R_{p'}h_{\lambda}\right)^{\frac{p-p_{0}}{p-1}}\right)^{\frac{1}{p_{0}}}\\
 & \leq c\left[\left(R_{p'}h_{\lambda}\right)^{\frac{p-p_{0}}{p-1}}\right]_{A_{p_{0}}}^{\beta}\left(\int_{\mathbb{R}^{n}}\left|f\right|^{p_{0}}\left(R_{p'}h_{\lambda}\right)^{\frac{p-p_{0}}{p-1}}dx\right)^{\frac{1}{p_{0}}}\\
\text{Hölder} & \leq c\left[\left(R_{p'}h_{\lambda}\right)^{\frac{p-p_{0}}{p-1}}\right]_{A_{p_{0}}}^{\beta}\left(\int_{\mathbb{R}^{n}}\left|f\right|^{p}dx\right)^{\frac{1}{p}}\left(\int_{\mathbb{R}^{n}}\left(R_{p'}h_{\lambda}\right)^{p'}dx\right)^{\frac{1}{p'}\frac{p-p_{0}}{p_{0}(p-1)}}\\
 & \leq c\left[\left(R_{p'}h_{\lambda}\right)^{\frac{p-p_{0}}{p-1}}\right]_{A_{p_{0}}}^{\beta}\left(\int_{\mathbb{R}^{n}}\left|f\right|^{p}dx\right)^{\frac{1}{p}}\\
\text{Jensen} & \leq c\left[R_{p'}h_{\lambda}\right]_{A_{1}}^{\beta\frac{p-p_{0}}{p-1}}\left(\int_{\mathbb{R}^{n}}\left|f\right|^{p}dx\right)^{\frac{1}{p}}\\
 & \leq c\left\Vert M\right\Vert _{L^{p'}(\mathbb{R}^{n})}^{\beta\frac{p-p_{0}}{p-1}}\left(\int_{\mathbb{R}^{n}}\left|f\right|^{p}dx\right)^{\frac{1}{p}}.
\end{split}
\]
Then we have that
\[
\lambda\left(\int_{\left\{ x\in\mathbb{R}^{n}\,:\,|Tf(x)|>\lambda\right\} }\right)^{\frac{1}{p}}\leq c\left\Vert M\right\Vert _{L^{p'}(\mathbb{R}^{n})}^{\beta\frac{p-p_{0}}{p-1}}\left(\int_{\mathbb{R}^{n}}\left|f\right|^{p}dx\right)^{\frac{1}{p}}
\]
and consequently
\[
\|Tf\|_{L^{p,\infty}(\mathbb{R}^{n})}\leq c\left\Vert M\right\Vert _{L^{p'}(\mathbb{R}^{n})}^{\beta\frac{p-p_{0}}{p-1}}\|f\|_{L^{p}(\mathbb{R}^{n})}.
\]
To finish the proof we recall that, for large $p>>p_{0}$, we have
that $\|M\|_{L^{p'}}\sim p.$ Therefore, we have
that 
\[
\|Tf\|_{L^{p,\infty}(\mathbb{R}^{n})}\leq cp^{\beta\frac{p-p_{0}}{p-1}}\leq cp^{\beta}.
\]
Since $p>>p_{0}$ we have that, dividing by $p^{\gamma_{T}-\varepsilon}$
and taking upper limits, we obtain 
\[
\infty=\limsup_{p\rightarrow\infty}\frac{\|Tf\|_{L^{p,\infty}(\mathbb{R}^{n})}}{p^{\gamma_{T}-\varepsilon}}\leq c\limsup_{p\rightarrow\infty}p^{\beta-\gamma_{T}+\varepsilon}.
\]
Consequently $\beta\geq\gamma_{T}$ and we're done.

\section{Proof of Lemma \ref{LemUnweighted}\label{sec:LemUnweighted}}

\subsection{Lemmata} In order to prove the unweighted estimates we need some lemmas. We present first some of the of the main ingredients of those results. We recall that the sharp maximal function $M_s^\sharp f$ by
\[M_s^\sharp f(x)=\sup_{Q \ni x}\left( \frac{1}{|Q|}\int_Q|f-f_Q|^s\right)^{\frac{1}{s}} \qquad 0<s<\infty\] 
In the case in which the supremum is taken only over dyadic cubes we write $M_s^{\sharp,d}$. We note that $M_s^\sharp$ is comparable so replacing one by the other when dealing with norm estimates will not make a difference for us. Analogously we will denote $M_s(f)=M(|f|^s)^\frac{1}{s}$.

Given a measurable function $f$ we define its non increasing rearrangement by
\[f^*(t)=\inf\{\lambda>0 : d_f(\lambda)\le t\}\]
where $d_f(\lambda)=|\{ x \in\mathbb{R}^n : |f(x)|>\lambda\}|$.

\begin{lem}
Let $0<\delta,\gamma<1$. There exists a constant $c=c_{n,\gamma,\delta}$
such that for any measurable function 
\begin{equation} 
f^{*}(t)\leq c\left(M_{\delta}^{\sharp}f\right)^{*}(\gamma t)+f^{*}(2t)\qquad t>0.\label{eq:2.5.2}
\end{equation}
\end{lem}

These type of estimates in this context go back to the work of R. Bagby and D. Kurtz in the mid 80s (see \cite{BK1} and \cite{BK2}). The proof of the Lemma can be found in \cite{Pe1} in the context of  $A_{p}$ weights and in \cite{OPR}  in the context of $A_{\infty}$  weights.  As a consequence we have the following.

\begin{lem}
\label{fFsharp}Let $1\leq p<\infty$, and $0<\delta<1$. Then there
exists a cosntant $c=c_{n,\delta}$ such that 
\[
\|f\|_{L^{p,\infty}}\leq cp\left\Vert M_{\delta}^{\sharp,d}f\right\Vert _{L^{p,\infty}}
\]
for each function $f$ such that $\left|\left\{ x\ :\ |f(x)|>t\right\} \right|<\infty$  for every $t>0$.
\end{lem}
\begin{proof}
Iterating \ref{eq:2.5.2} we have that 
\[
\begin{split}f^{*}(t) & \leq c\sum_{k=0}^{\infty}\left(M_{\delta}^{\sharp}f\right)^{*}(2^{k}\gamma t)+f^{*}(+\infty)\\
 & \leq c\frac{1}{\log2}\int_{t\frac{\gamma}{2}}^{\infty}\left(M_{\delta}^{\sharp}f\right)^{*}(s)\frac{ds}{s}
\end{split}
\]
using that $f^{*}(+\infty)=0$ which follows since $\left|\left\{ x\ :\ |f(x)|>t\right\} \right|<\infty$ for each $t>0$. Now we recall that 
\[
Sf(x)=\int_{x}^{\infty}f(s)\frac{1}{s}ds
\]
is the adjoint of Hardy operator. Then the preceding estimate can
be restated as follows
\begin{equation}
f^{*}(t)\leq c\frac{1}{\log2}S\left(\left(M_{\delta}^{\sharp}f\right)^{*}\right)\left(t\frac{\gamma}{2}\right).\label{eq:f*}
\end{equation}
Now we see that 
\begin{equation}
\left\Vert S\left(g^{*}\right)\right\Vert _{L^{p,\infty}}\leq p\left\Vert g\right\Vert _{L^{p,\infty}}.\label{eq:Sf*}
\end{equation}
Indeed, since $Sg^{*}(x)$ is decreasing,
\[
\left\Vert S\left(g^{*}\right)\right\Vert _{L^{p,\infty}}=\sup_{t>0}t^{\frac{1}{p}}S\left(g^{*}\right)^{*}(t)=\sup_{t>0}t^{\frac{1}{p}}S\left(g^{*}\right)(t),
\]
and the desired estimate follows from observing that for each $t>0$ we have that 
\[
\begin{split} t^{\frac{1}{p}}S\left(g^{*}\right)(t) & =t^{\frac{1}{p}}\int_{t}^{\infty}g^{*}(s)\frac{1}{s}\,ds=t^{\frac{1}{p}}\int_{t}^{\infty}s^{\frac{1}{p}}g^{*}(s)\frac{1}{s^{1+\frac{1}{p}}}ds\\
 & \leq\|g\|_{L^{p,\infty}}t^{\frac{1}{p}}\int_{t}^{\infty}\frac{1}{s^{1+\frac{1}{p}}}ds=\|g\|_{L^{p,\infty}}t^{\frac{1}{p}}\int_{t}^{\infty}\frac{1}{s^{1+\frac{1}{p}}}ds\\
 & =\|g\|_{L^{p,\infty}}t^{\frac{1}{p}}\left[\frac{1}{-\frac{1}{p}s^{\frac{1}{p}}}\right]_{s=t}^{\infty}=p\|g\|_{L^{p,\infty}}.
\end{split}
\]
Armed with  (\ref{eq:f*}) and (\ref{eq:Sf*}) we can now
establish the desired inequality:
\[
\begin{split}\|f\|_{L^{p,\infty}} & =\sup_{t>0}t^{\frac{1}{p}}f^{*}(t)\leq c\frac{1}{\log2}\sup_{t>0}t^{\frac{1}{p}}S\left(\left(M_{\delta}^{\sharp}f\right)^{*}\right)\left(t\frac{\gamma}{2}\right)\\
 & =c\frac{1}{\log2}\left(\frac{2}{\gamma}\right)^{\frac{1}{p}}\sup_{t>0}\left(t\frac{\gamma}{2}\right)^{\frac{1}{p}}S\left(\left(M_{\delta}^{\sharp}f\right)^{*}\right)\left(t\frac{\gamma}{2}\right)\\
 & =c\frac{1}{\log2}\left(\frac{2}{\gamma}\right)^{\frac{1}{p}}\left\Vert S\left(\left(M_{\delta}^{\sharp}f\right)^{*}\right)\right\Vert _{L^{p,\infty}}\\
 & \leq c\frac{1}{\log2}\left(\frac{2}{\gamma}\right)^{\frac{1}{p}}p\|M_{\delta}^{\sharp}f\|_{L^{p,\infty}}
\end{split}
\]
\end{proof}
\begin{lem}
\label{MMsharp}Let $1\leq p<\infty$ and $0<\varepsilon\leq1$. Suppose
that $f$ is a function such that for every $t>0$ $|\{x\ :\ |f(x)|>t\}|<\infty$.
Then there exists a constant $c=c_{n,\varepsilon}$ such that 
\[
\left\Vert M_{\varepsilon}^{d}f\right\Vert _{L^{p,\infty}}\leq cp\left\Vert M_{\varepsilon}^{\sharp,d}f\right\Vert _{L^{p,\infty}}
\]
\end{lem}
\begin{proof}

We apply the preceding lemma with $f$ replaced by $M_{\varepsilon}f$ and $\delta=\varepsilon_{0}$
such that $0<\varepsilon_{0}<\varepsilon<1$. Then 
\[
\|M_{\varepsilon}f\|_{L^{p,\infty}}\leq cp\left\Vert M_{\varepsilon_{0}}^{\sharp,d}\left(M_{\varepsilon}f\right)\right\Vert _{L^{p,\infty}}.
\]
We also know (see \cite{O}) that 
\[
M_{\varepsilon_{0}}^{\sharp,d}\left(M_{\varepsilon}f\right)\leq cM_{\varepsilon}^{\sharp,d}f.
\]
Consequently, 
\[
\left\Vert M_{\varepsilon_{0}}^{\sharp,d}\left(M_{\varepsilon}f\right)\right\Vert _{L^{p,\infty}}\leq c\left\Vert M_{\varepsilon}^{\sharp,d}f\right\Vert _{L^{p,\infty}}.
\]
This concludes the proof of the lemma. 
\end{proof}
\subsection{Proof of Lemma \ref{LemUnweighted}}
Armed with the preceding results we are in the position to establish
Lemma \ref{LemUnweighted}. We consider different cases.

\subsection*{Calderón-Zygmund operators}

Firstly we obtain the upper bound. Using lemma \ref{fFsharp}
\[
\|Tf\|_{L^{p,\infty}}\leq cp\left\Vert M_{\delta}^{\sharp}\left(Tf\right)\right\Vert _{L^{p,\infty}}\leq   cp\|Mf\|_{L^{p,\infty}}\leq  cp\|f\|_{L^{p}},
\]
since $M_{\delta}^{\sharp}\left(Tf\right) \leq c_{\delta} Mf$, $0<\delta<1$ as can be found \cite{AP}.
Now we deal with the lower bound. It is well known that
\[
H\left(\chi_{[0,1]}\right)(x)= \log\left(\frac{|x|}{|x-1|}\right).
\]
Then 
\[
\begin{split}\|H\chi_{(0,1)}\|_{L^{p,\infty}} & =\sup_{t>0}t\left|\left\{ x\in\left(0,\frac{1}{2}\right)\ :\ -\frac{1}{\pi}\log\left(\frac{x}{x-1}\right)>t\right\} \right|^{\frac{1}{p}}\\
 & =\sup_{t>0}t\left|\left\{ x\in\left(0,\frac{1}{2}\right)\ :\ \Phi(x)>t\right\} \right|^{\frac{1}{p}}\\
{\scriptstyle \Phi\text{ decreasing}} & =\sup_{t>0}t\left|\left\{ x\in\left(0,\frac{1}{2}\right)\ :\ \Phi^{-1}(t)>x\right\} \right|^{\frac{1}{p}}=\sup_{t>0}t\Phi^{-1}(t)^{\frac{1}{p}}
\end{split}
\]
Using again the properties of $\Phi$ 
\[
\sup_{t>0}t\Phi^{-1}(t)^{\frac{1}{p}}=\sup_{0<x<\frac{1}{2}}\Phi(x)x^{\frac{1}{p}}
\]
Now we observe that for every $0<x<\frac{1}{2}$, 
\[
\Phi(x)x^{\frac{1}{p}}\geq-\frac{1}{\pi}\log\left(\frac{x}{2}\right)x^{\frac{1}{p}}=-p2^{\frac{1}{p}}\frac{1}{\pi}\log\left(\left(\frac{x}{2}\right)^{\frac{1}{p}}\right)\left(\frac{x}{2}\right)^{\frac{1}{p}}\geq cp
\]
and we're done.

\subsection*{Commutators}
Firstly we obtain the upper bound. Suppose that $\|b\|_{BMO}=1$.
Then using lemma \ref{fFsharp}
\[
\begin{split}\|[b,T]f\|_{L^{p,\infty}} & \leq cp\left\Vert M_{\delta}^{\sharp}\left([b,T]f\right)\right\Vert _{L^{p,\infty}}\leq cp\left\Vert M_{\varepsilon}\left(Tf\right)\right\Vert _{L^{p,\infty}}+cp\left\Vert M^{2}f\right\Vert _{L^{p,\infty}}\\
 & =cp\left(L_{1}+L_{2}\right)
\end{split}
\]
Now we observe that using Lemma \ref{MMsharp} with $0<\varepsilon<1$,
\[
L_{1}=\left\Vert M_{\varepsilon}\left(Tf\right)\right\Vert _{L^{p,\infty}}\leq cp\left\Vert M_{\varepsilon}^{\sharp}(Tf)\right\Vert _{L^{p,\infty}}\leq cp\|Mf\|_{L^{p,\infty}}\leq cp\|f\|_{L^{p}}.
\]
For $L_{2}$ we have
\[
L_{2}=\left\Vert M^{2}f\right\Vert _{L^{p,\infty}}\leq c\left\Vert Mf\right\Vert _{L^{p}}\leq cp'\|f\|_{L^{p}}.
\]
Consequently
\[
\|[b,T]f\|_{L^{p,\infty}}\leq cp^{2}p'\|f\|_{L^{p}}.
\]

Let us focus on the lower bound. Consider the Hilbert transform 
\[
Hf(x)=pv\,\int_{\mathbb{R}}\frac{f(y)}{x-y}\,dy,
\]
and consider the BMO function $b(x)=\log|x|$. Let $f=\chi_{(0,1)}$.
If $0<x<1$, 
\[
[b,H]f(x)=\int_{0}^{1}\frac{\log(x)-\log(y)}{x-y}\,dy=\int_{0}^{1}\frac{\log(\frac{x}{y})}{x-y}\,dy=\int_{0}^{1/x}\frac{\log(\frac{1}{t})}{1-t}\,dt
\]
\[
\int_{0}^{1/x}\frac{\log(\frac{1}{t})}{1-t}\,dt=\int_{0}^{1}\frac{\log(\frac{1}{t})}{1-t}\,dt+\int_{1}^{1/x}\frac{\log(\frac{1}{t})}{1-t}\,dt
\]
$\frac{\log(\frac{1}{t})}{1-t}$ is positive for $(0,1)\cup(1,\infty)$
we have for $0<x<1$ 
\[
|[b,H]f(x)|>\int_{1}^{1/x}\frac{\log(\frac{1}{t})}{1-t}\,dt.
\]
But since 
\[
\int_{1}^{\infty}\frac{\log t}{t-1}\,dt=\infty
\]
and 
\[
\lim_{L\to\infty}\frac{\int_{1}^{L}\frac{\log(\frac{1}{t})}{1-t}\,dt}{(\log L)^{2}}=1
\]
we have that for some $x_{0}<1$ 
\[
|[b,H]f(x)|>c\,(\log\frac{1}{x})^{2}\qquad0<x<x_{0}.
\]
 and then 
\[
\begin{split}t|\{x\in\mathbb{R}:|[b,H]f(x)|>t\}|^{\frac{1}{p}} & \geq t\left|\left\{ x\in(0,x_{0}):\,c\left(\log\frac{1}{x}\right)^{2}>t\right\} \right|^{\frac{1}{p}}\\
 & \geq cte^{\frac{-\sqrt{t}}{p}}=cp^{2}\frac{t}{p^{2}}e^{-\sqrt{\frac{t}{p^{2}}}}.
\end{split}
\]
We have that
\[
\|\left[b,H\right]f\|_{L^{p,\infty}(\mathbb{R})}=\sup_{t>0}t|\{x\in\mathbb{R}:|[b,H]f(x)|>t\}|^{\frac{1}{p}}\geq\sup_{t>0}cp^{2}\frac{t}{p^{2}}e^{-\sqrt{\frac{t}{p^{2}}}}\geq cp^{2}.
\]
Let $b=\log|x|$ and $f(x)=\chi_{(0,1)}(x)$. Now if $x>e$, we have
that 
\[
\begin{split}\left|\left[b,H\right]f(x)\right| & =\int_{0}^{1}\frac{\log\left(x\right)-\log\left(y\right)}{x-y}dy=\int_{0}^{\frac{1}{x}}\frac{\log\left(\frac{1}{t}\right)}{1-t}tdt\\
 & \geq\log(x)\int_{0}^{\frac{1}{x}}\frac{1}{1-t}dt=\frac{\log(x)}{x}.
\end{split}
\]
Now we observe that 
\[
\begin{split}\|\left[b,H\right]f\|_{L^{p,\infty}(\mathbb{R})} & =\sup_{t>0}t\left|\left\{ x\in\mathbb{R}\,:\,\left|\left[b,H\right]f(x)\right|>t\right\} \right|^{\frac{1}{p}}\geq\sup_{t>0}t\left|\left\{ x>e\,:\,p^{2}\frac{\log(x)}{x}>t\right\} \right|^{\frac{1}{p}}\\
 & =\sup_{t>0}t\left|\left\{ x>e\,:\,\Phi(x)>t\right\} \right|^{\frac{1}{p}}
\end{split}
\]
$\Phi$ is decreasing so 
\[
\begin{split} & \sup_{t>0}t\left|\left\{ x>e\,:\,\Phi(x)>t\right\} \right|^{\frac{1}{p}}=\sup_{t>0}t\left|\left\{ x>e\,:\,\Phi^{-1}(t)>x\right\} \right|^{\frac{1}{p}}\\
 & =\sup_{t>0}t\left(\Phi^{-1}(t)-e\right)^{\frac{1}{p}}=\sup_{x>e}\Phi(x)\left(x-e\right)^{\frac{1}{p}}\geq\sup_{x>2e}\Phi(x)\left(x-e\right)^{\frac{1}{p}}
\end{split}
\]
Now we observe that
\[
\begin{split}\Phi(x)\left(x-e\right)^{\frac{1}{p}} & \geq-\log\left(\frac{x-1}{x}\right)\log(x)\left(x-\frac{x}{2}\right)^{\frac{1}{p}}=\frac{\log(x)}{x^{\frac{1}{p'}}}\left(\frac{1}{2}\right)^{\frac{1}{p}}\\
 & =p'\frac{\log\left(x^{\frac{1}{p'}}\right)}{x^{\frac{1}{p'}}}\left(\frac{1}{2}\right)^{\frac{1}{p}}\geq\frac{1}{2}p'\frac{\log\left(x^{\frac{1}{p'}}\right)}{x^{\frac{1}{p'}}}
\end{split}
\]
and consequently 
\[
\|\left[b,H\right]f\|_{L^{p,\infty}(\mathbb{R})}\geq\frac{1}{2}p'
\]

\subsection*{Maximal function}

For the upper bound
\[
\|M^{k}f\|_{L^{p,\infty}}\leq c_{n}\|M^{k-1}f\|_{L^{p}}\leq c_{n}(p')^{k-1}\|f\|_{L^{p}}.
\]
For the lower bound firstly we observe that $M^{k}\chi_{(0,1)}\simeq\frac{(\log|x|)^{k-1}}{|x|}$
\[
\begin{split}\|M^{k}\chi_{(0,1)}\|_{L^{p,\infty}} & \simeq\sup_{t>0}t\left|\left\{ x\in\mathbb{R}\,:\,\frac{(\log|x|)^{k-1}}{|x|}>t\right\} \right|^{\frac{1}{p}}\geq\sup_{t>0}t\left|\left\{ x>e\,:\,\frac{(\log x)^{k-1}}{x}>t\right\} \right|^{\frac{1}{p}}\\
 & =\sup_{t>0}t\left|\left\{ x>e\,:\,\Psi(x)>t\right\} \right|^{\frac{1}{p}}=\sup_{t>0}t\left|\left\{ x>e\,:\,x<\Psi^{-1}(t)\right\} \right|^{\frac{1}{p}}\\
 & =\sup_{t>0}t\left(\Psi^{-1}(t)-e\right)^{\frac{1}{p}}=\sup_{x>e}\frac{(\log x)^{k-1}}{x}(x-e)^{\frac{1}{p}}\geq\sup_{x>2e}\frac{(\log x)^{k-1}}{x}(x-e)^{\frac{1}{p}}
\end{split}
\]
Now we see that
\[
\begin{split}\frac{(\log x)^{k-1}}{x}(x-e)^{\frac{1}{p}} & \geq\frac{(\log x)^{k-1}}{x}(x-\frac{x}{2})^{\frac{1}{p}}\geq\left(\frac{1}{2}\right)^{\frac{1}{p}}\frac{(\log x)^{k-1}}{x^{\frac{1}{p'}}}\\
 & =\left(\frac{1}{2}\right)^{\frac{1}{p}}\left(p'\right)^{k-1}\frac{\log(x^{\frac{1}{p'}})^{k-1}}{x^{\frac{1}{p'}}}\geq c\left(\frac{1}{2}\right)^{\frac{1}{p}}\left(p'\right)^{k-1}
\end{split}
\]
This gives the desired estimate.

\section{Proof of Theorem \ref{Thm:ApBounds}\label{sec:ThmApBounds}}
We consider each case separatedly. 
\subsection*{Calderón-Zygmund operators}

It was established in \cite{HLMORSUT} that
\[
\|T\|_{L^{p}(w)\rightarrow L^{p,\infty}(w)}\leq c_{n}[w]_{A_{p}}
\]
and the optimality of the exponent is a direct corollary of the combination
of Theorem \ref{Thm:LowerBoundWE} and Lemma \ref{LemUnweighted}.

\subsection*{Commutators}

The lower bound of the exponent is again a direct corollary of the
combination of Theorem \ref{Thm:LowerBoundWE} and Lemma \ref{LemUnweighted}.
For the upper exponent we are going to use a proof based on 
a sparse domination result obtained in \cite{LORR}.

We recall that a family of dyadic cubes $\mathcal{S}$ is $\eta$-sparse
with $\eta\in(0,1)$ if for each cube $Q\in\mathcal{S}$ there exists
a measurable subset $E_{Q}\subset Q$ such that $E_{Q}$ are pairwise
disjoint and $\eta|Q|\leq|E_{Q}|$.

The following result is well known
\begin{thm}
\label{Thm:HyLi}Let $1<p<\infty$. Then if $w\in A_{p}$
\[\begin{split}
\|\mathcal{A}_{\mathcal{S}}\|_{L^{p}(w)\rightarrow L^{p}(w)}&\le c_{n,p}[w]_{A_{p}}^{\max\left\{1,\frac{1}{p-1}\right\} },\\
\|\mathcal{A}_{\mathcal{S}}\|_{L^{p}(w)\rightarrow L^{p,\infty}(w)}&\le c_{n,p}[w]_{A_{p}}\end{split}
\]
where $\mathcal{A}_{\mathcal{S}}(f)=\sum_{Q\in\mathcal{S}}\frac{1}{|Q|}\int_{Q}f\chi_{Q}(x)$.
\end{thm}
In \cite{LORR} it was proved that that commutators can be controlled by suitable sparse operators. 
The precise statement is the following.
\begin{thm} Let $T$ a Calder\'on-Zygmund operator, $b\in L^1_\text{loc}(\mathbb{R}^n)$ and $f\in\mathcal{C}_c^\infty$. There exist $3^{n}$ dyadic lattices $\mathcal{D}_{j}$
and $3^{n}$ sparse families $\mathcal{S}_{j}\subset\mathcal{D}_{j}$
such that 
\[
|[b,T]f(x)|\leq c_{n}c_{T}\sum_{j=1}^{3^{n}}\left(\mathcal{T}_{b,\mathcal{S}_{j}}f(x)+\mathcal{T}_{b,\mathcal{S}_{j}}^{*}f(x)\right)
\]
where $c_{T}=c_{K}+c_\delta+\|T\|_{L^{2}\rightarrow L^{2}}$
and 
\[
\begin{split}\mathcal{T}_{b,\mathcal{S}}f(x) & =\sum_{Q\in\mathcal{S}}|b(x)-b_{Q}|\frac{1}{|Q|}\int_{Q}|f|\chi_{Q}(x),\\
\mathcal{T}_{b,\mathcal{S}}^{*}f(x) & =\sum_{Q\in\mathcal{S}}\frac{1}{|Q|}\int_{Q}|b-b_{Q}||f|\chi_{Q}(x).
\end{split}
\]
\end{thm}
Consequently it suffices to establish the weak-type $(p,p)$ for those sparse
operators. Without loss of generality we may assume that $\|b\|_{BMO}=1$. 

We deal first with $\mathcal{T}_{b,\mathcal{S}}f$. We observe that
\[
\|\mathcal{T}_{b,\mathcal{S}}f\|_{L^{p,\infty}(w)}=\sup_{\|g\|_{L^{p',1}=1}}\left|\int\mathcal{T}_{b,\mathcal{S}}(f)gw\right|
\]
Now 
\[
\left|\int\mathcal{T}_{b,\mathcal{S}}(f)gw\right|\leq\int\mathcal{T}_{b,\mathcal{S}}(f)|g|w=\sum_{Q\in\mathcal{S}}\frac{1}{|Q|}\int_{Q}|f|\int_{Q}|b-b_{Q}||g|w
\]
Arguing as in the proof of \cite[Theorem 1.4]{LORR} we have that there exists a sparse
family $\tilde{\mathcal{S}}\supset\mathcal{S}$ such that for every
cube in $Q\in\tilde{\mathcal{S}}$, 
\[
\int_{Q}|b-b_{Q}||g|w\leq c_{n}\|b\|_{BMO}\int_{Q}\mathcal{A}_{\tilde{\mathcal{S}}}(|g|w).
\]
Then, taking also into account that $\mathcal{A}_{\tilde{\mathcal{S}}}$
is self-adjoint, 
\[
\begin{split}\sum_{Q\in\mathcal{S}}\frac{1}{|Q|}\int_{Q}|f|\int_{Q}|b-b_{Q}||g|w & \leq c_{n}\|b\|_{BMO}\sum_{Q\in\mathcal{\tilde{S}}}\frac{1}{|Q|}\int_{Q}|f|\int_{Q}\mathcal{A}_{\tilde{\mathcal{S}}}(|g|w)\\
&=c_n\|b\|_{BMO}\int\mathcal{A}_{\tilde{\mathcal{S}}}(f)\mathcal{A}_{\tilde{\mathcal{S}}}\left(|g|w\right)\\
 & =c_n\|b\|_{BMO}\int\left(\mathcal{A}_{\tilde{\mathcal{S}}}\circ\mathcal{A}_{\tilde{\mathcal{S}}}\right)(f)|g|w
\end{split}
\]
and we have that 
\[
\begin{split}\|\mathcal{T}_{b,\mathcal{S}}f\|_{L^{p,\infty}(w)} & =\sup_{\|g\|_{L^{p',1}=1}}\left|\int\mathcal{T}_{b,\mathcal{S}}(f)gw\right|\leq\sup_{\|g\|_{L^{p',1}=1}}\int\left(\mathcal{A}_{\tilde{\mathcal{S}}}\circ\mathcal{A}_{\tilde{\mathcal{S}}}\right)(f)|g|w\\
 & \leq c_{n}\sup_{\|g\|_{L^{p',1}=1}}\int(\mathcal{A}_{\tilde{\mathcal{S}}}\circ\mathcal{A}_{\tilde{\mathcal{S}}})(f)|g|w\leq c_{n}\|(\mathcal{A}_{\tilde{\mathcal{S}}}\circ\mathcal{A}_{\tilde{\mathcal{S}}})f\|_{L^{p,\infty}(w)}\|g\|_{L^{p',1}(w)}.
\end{split}
\]
The  desired inequality follows from applying twice Theorem \ref{Thm:HyLi}.

For $\mathcal{T}^*_{b,\mathcal{S}}$ we can argue analogously. 
Indeed, it is clear that we also have that
\[
\int_{Q}|b-b_{Q}||f|w\leq c_{n}\|b\|_{BMO}\int_{Q}\mathcal{A}_{\tilde{\mathcal{S}}}(|f|w)
\]
and this yields
\[\mathcal{T}^*_{b,\mathcal{S}}(f)\leq c_n\|b\|_{BMO}\left(\mathcal{A}_{\tilde{\mathcal{S}}}\circ \mathcal{A}_{\tilde{\mathcal{S}}}\right)(|f|).\] 
Now it suffices to use Theorem \ref{Thm:HyLi} to end the proof.

\subsection*{Maximal operator}
The lower bound of the exponent is a straightforward consequence of
Theorem \ref{Thm:LowerBoundWE} and Lemma \ref{LemUnweighted}. For the
upper bound it readily follows from (\ref{estMax}) that
\[
\|M^{k}f\|_{L^{p,\infty}(w)}\leq c_{n}[w]_{A_{p}}^{\frac{1}{p}}\|M^{k-1}f\|_{L^{p}(w)}\leq c_{n}[w]_{A_{p}}^{\frac{1}{p}+\frac{k-1}{p-1}}\|f\|_{L^{p}(w)}
\]
and we are done.

\end{document}